\documentclass[12pt]{amsart}
\usepackage[active]{srcltx}
\usepackage{a4wide}
\usepackage{amsthm,amsfonts,amsmath,mathrsfs,amssymb}
\usepackage{dsfont}
\usepackage[T1]{fontenc}
\usepackage[utf8]{inputenc}
\usepackage{pdfpages}
\usepackage{graphicx}

\def\a{\alpha}               
            
       \def\t{\theta}       
         \def\r{\rho}

\def\e{\varepsilon}       \def\vf{\varphi}

\def\D{{\mathbb D}}     \def\T{{\mathbb T}}

\def\({\left(}       \def\){\right)}

\DeclareMathOperator{\re}{Re}

\newtheorem{prop}{\sc Proposition}
\newtheorem{lem}[prop]{\sc Lemma}
\newtheorem{thm}[prop]{\sc Theorem}

\begin{document}
\title[Non-normable spaces of analytic functions]{Non-normable spaces of analytic functions}
\author[I. Jim\'enez]{Iv\'an Jim\'enez}
\address{Departamento de Matem\'aticas, Universidad Aut\'onoma de
Madrid, 28049 Madrid, Spain}
\email{ivan.jimenezs@estudiante.uam.es}
\author[D. Vukoti\'c]{Dragan Vukoti\'c}
\address{Departamento de Matem\'aticas, Universidad Aut\'onoma de
Madrid, 28049 Madrid, Spain} \email{dragan.vukotic@uam.es}
\date{01 August, 2024.}

\begin{abstract}
For each value of $p$ such that $0<p<1$, we give a specific example of two functions in the Hardy space $H^p$ and in the Bergman space $A^p$ that do not satisfy the triangle inequality. For Hardy spaces, this provides a much simpler proof than the one due to Livingston that involves abstract functional analysis arguments and an approximation theorem. For Bergman spaces, we have not been able to locate any examples or proofs in the existing literature.
\end{abstract}
\maketitle
\section{Introduction}
 \label{sect-intro}
\par
It is well known that, for a given positive measure $\mu$, the Lebesgue space $L^p(X,\mu)$ is a complete (Banach) space when equipped with the usual norm $\|f\|_p = \( \int_X |f|^p\,d\mu\)^{1/p}$ whenever $1\le p<\infty$, while the above expression $\|\cdot\|_p$ in general does not define a norm when $0<p<1$ since it fails to satisfy the triangle inequality.
\par
The standard Hardy and Bergman spaces of analytic functions in the unit disk $\D$, denoted respectively by $H^p$ and $A^p$, can be seen as closed subspaces of the spaces $L^p(\T,dm)$ and $L^p(\D,dA)$, where $\T$ denotes the unit circle, $dm(\t)=d\t/(2\pi)$ the normalized arc length measure on it, and $dA(z)=dx\,dy/\pi$ the normalized area measure on $\D$. These spaces are also complete with respect to their respective $L^p$-type norms when $1\le p<\infty$, and it is also known that they are not normed space when $0<p<1$. There are many known  monographs or texts that treat Hardy spaces or Bergman spaces \cite{DrS, D, DS, F, G, HKZ, JVA, K, R1} and this fact is mentioned in passing in many of them.
\par
However, it seems that the proof of this ``obvious'' fact is not contained in any of the texts quoted, not even among the exercises. The likely reason is that this is not so easy to prove in a direct way. The spaces $H^p$ and $A^p$ consist of holomorphic functions, which have many rigidity properties and therefore they cannot be varied in any flexible way, not even on very small sets. Thus, specific examples are not nearly as easy to construct as in the context of $L^p$ spaces where we have all the freedom of modifying measurable functions at our pleasure. Also, it is complicated in general to compute or estimate precisely the norms of functions in such spaces when $p\neq 2$.
\par
This specific issue was discussed from a different point of view in Livingston's paper \cite{L} from the 1950's. By a well-known theorem of Kolmogorov from 1934 \cite[Theorem~1.39]{R2}, a topological vector space is normable (has an equivalent normed topology) if and only if its origin has a convex bounded neighborhood. It was shown in \cite{L} that, for $0<p<1$, the open unit ball of $H^p$ contains no convex neighborhood of the origin, which implies the  non-normability of the space. To the best of our knowledge, we have not been able to locate a proof of the analogous well-known fact for $A^p$ spaces in the literature.
\par
It seems useful to have a `hard analysis' proof that the usual expression $\|\cdot\|_p$ is not a norm when  $0<p<1$; of course, a  `soft analysis' argument seems to be called for in order to prove the  stronger fact that actually no norm can be defined on the same space defining an equivalent topology. The purpose of this note is to fill this gap in the literature by giving specific examples of two functions, in both $H^p$ and $A^p$ spaces with $0<p<1$, that do not satisfy the triangle inequality. We hope that graduate students and other young researchers may find useful the examples given in this note.

\section{Preliminary facts}
 \label{sect-prelim}
\par
\subsection{Hardy spaces}
\label{subsect-Hardy}
\par
Let $\D$ denote the unit disc in the complex plane. It is well known \cite[Chapter~1]{D} that for any function $f$ analytic in $\D$ and  $0<p<\infty$, the integral means of order $p$ of $f$:
$$
M_p(r;f) = \( \int_0^{2\pi} |f(re^{i\t})|^p \frac{d\t}{2\pi} \)^{1/p}
$$
are increasing functions of $r\in (0,1)$. The Hardy space $H^p$ is the set of all analytic functions in $\D$ for which these means have finite limits: $\|f\|_{H^p}=\lim_{r\to 1^-}M_p(r;f)<\infty$. This is not a true norm if $0<p<1$ but the same notation and the term ``norm'' will still be used in this case. We list several properties of Hardy spaces that will be needed in the sequel.
\par
It is a well-known fact that $H^p$ functions have radial limits: $\tilde{f}(e^{i \t})=\lim_{r\to 1^-} f(r e^{i\t})$ almost everywhere on the unit circle and the norm can be computed in terms of these limits (see \cite{D}, \cite{F}, or \cite{G}):
$$
 \|f\|_{H^p} = \(\int_0^{2\pi} |\tilde{f}(e^{i\t})|^p \frac{d\t}{2\pi} \)^{1/p}\,.
$$
\par
A direct computation shows that the $H^p$ norm is invariant under rotations; in particular, if $g(z)=f(-z)$, then $\|g\|_{H^p}=\|f\|_{H^p}$.
\par
It is also well known that the composition with the function $z\mapsto z^2$ is bounded on any Hardy space. The following useful fact quantifies this.
\begin{lem}\label{lem-cvh}
Let $f\in H^p$ and $h(z)=f(z^2)$, for $z\in\D$. Then $h\in H^p$ and
$\|h\|_{H^p}=\|f\|_{H^p}$.
\end{lem}
\begin{proof}
Follows by a simple change of variable $t=2\t$ and periodicity:
$$
 \|h\|_{H^p}^p=\int_0^{2\pi} |\tilde{f}(e^{2i\t})|^p \frac{d\t}{2\pi} = \frac12 \int_0^{4\pi} |\tilde{f}(e^{i t})|^p \frac{d t}{2\pi} = \int_0^{2\pi} |\tilde{f}(e^{i t})|^p \frac{d t}{2\pi}=\|f\|_{H^p}^p\,.
$$
\end{proof}
\par
\subsection{Bergman spaces}
\label{subsect-Bergman}
\par
The Bergman ``norm'' is defined as
\begin{equation}
 \|f\|_{A^p}^p = \int_\D |f(z)|^p\,dA(z) = \int_0^1 2 r M_p^p(r,f) \,dr\,,
 \label{eqn-Ap-norm}
\end{equation}
where $dA$ denotes the normalized Lebesgue area measure on $\D$:
$$
 dA(z)=\frac{dx\,dy}{\pi}=\frac{r\,dr\,d\t}{\pi}\,, \qquad z=x+iy=re^{i \t}\,.
$$
In the most special case when $p=2$ and $f(z)=\sum_{n=0}^\infty a_n z^n\/$ in $\D$, using orthogonality, the norm can be computed explicitly in terms of the Taylor coefficients:
\begin{equation}
 \|f\|_{A^2}^2 = \sum_{n=0}^\infty \frac{|a_n|^2}{n+1}\,.
 \label{eqn-A2-norm}
\end{equation}
We refer the readers to \cite{DS} or \cite{HKZ} for these basic facts. \par
Of course, the expression in \eqref{eqn-Ap-norm} is a norm only when $1\le p<\infty$ but we shall again use the term ``norm'' also when $0<p<1$. To show explicitly that this is not a norm for small exponents $p$ becomes more involved than for the Hardy spaces. This is due to the lack of boundary values of functions in $A^p$ and to the fact that computing the norm involves integration with respect to area.
\par
It is readily checked that the Bergman ``norm'' is invariant under rotations, hence $f(z)$ and $f(-z)$ have the same norm. Again, it is well known that the composition with the function $z\mapsto z^2$ is bounded on any Bergman space. The following related exact formula will be useful.
\begin{lem}\label{lem-cv}
If $h\in A^p$, then
$$
 \int_\D |h(z)|^p\,dA(z) = 2 \int_\D |h(z^2)|^p |z|^2\,dA(z)\,.
$$
\end{lem}
\begin{proof}
By the obvious change of variable $2\t=\vf$ and periodicity, followed by another change of variable $r^2=\r$, we obtain
\begin{eqnarray*}
 \int_\D |h(z^2)|^p |z|^2\,dA(z) &=& \int_0^1 2 r^3 \int_0^{2\pi} |h(r^2 e^{2i\t})|^p \frac{d\t}{2\pi}\,dr
\\
 &=& \int_0^1 2 r^3 \int_0^{2\pi} |h(r^2 e^{i\vf})|^p \frac{d\vf}{2\pi}\,dr
\\
 &=& \int_0^1 \r M_p^p(\r,h) \,d\r\,,
\end{eqnarray*}
and the statement follows.
\end{proof}
The following fact is well known. We include an indication of a proof for the sake of completeness.
\begin{lem}\label{lem-Ap}
Let $h(z)=(1-z)^{-\a}$, $\a>0$. Then $h\in A^p$ if and only if $p \a<2$.
\end{lem}
\begin{proof}
This is easily established by integrating in polar coordinates centered at $z=1$ rather than at the origin: write $z=1-r e^{i\t}$, where $-\pi/2<\t<\pi/2$ and $0<r<2\cos \t$. The rest is also elementary calculus.
\end{proof}

\section{Examples}
 \label{sect-examples}
\par
\subsection{A Hardy space example}
\label{subsect-example}
It seems intuitively clear that the nice properties that $H^p$ spaces enjoy should allow us to present simple examples. Actually, it turns out that there is one single example that works for every $p$ with $0<p<1$.
\begin{thm}\label{thm-large-p}
Let $0<p<1$. Then the functions $f$ and $g$, defined respectively by
$$
 f(z)=\frac{1+z}{1-z}\,, \quad g(z)=-f(-z)=-\frac{1-z}{1+z}\,,
$$
both belong to $H^p$ but fail to satisfy the triangle inequality for $\|\cdot\|_{H^p}$.
\end{thm}
\begin{proof}
It is a well-known exercise \cite[Chapter~1, Problem~1]{D} that, for $h(z)=\frac{1}{1-z}$, we have $h\in H^p$ if and only if $0<p<1$, and the same is easily seen for the closely related function $f$. By the basic properties, $\|f\|_{H^p}=\|g\|_{H^p}$. Also, a direct computation shows that
$$
 f(z)+g(z)=\frac{4z}{1-z^2}\,.
$$
Taking into account that $|z|=1$ on the unit circle, as well as  Lemma~\ref{lem-cvh} and the inequality $|1+z|\le 2$ which is actually strict for all $z\in\T \setminus\{1\}$, we obtain that
$$
 \|f+g\|_{H^p}=\left\|\frac{4}{1-z^2}\right\|_{H^p}=4\left\|\frac{1}{1-z}\right\|_{H^p} > 2 \left\|\frac{1+z}{1-z}\right\|_{H^p}=2\|f\|_{H^p}=\|f\|_{H^p}+\|g\|_{H^p}\,,
$$
showing that the triangle inequality fails in this case.
\end{proof}
\par
Once discovered, the last example may even look trivial. However, it should be mentioned that it was not the first example of this kind that we found; the earlier examples required a lot more involved calculations and estimates. Moreover, the ``obvious'' example one would first think of does not work: for the function $h$ mentioned in the proof and the related modified function $k(z)=-h(-z)$, it can be easily checked by a completely similar argument that we have equality in the triangle inequality: $\|h+k\|_{H^p}=\|h\|_{H^p}+ \|k\|_{H^p}$.
\subsection{Bergman space examples}
\label{subsect-examples}
In this case, we do not have one single example covering the entire range of values of the exponent $p$. Instead, we exhibit two different types of examples, depending on the value of $p$. The first example is a modification of the Hardy space example given earlier.
\par
\begin{thm}\label{thm-large-p}
Let $\frac12\le p<1$ and let $\e \le 1$ and $\/(1-p)/p\le\e<2(1-p)/p$. Then the functions $f$ and $g$, given by
$$
 f(z)=\frac{(1+z)^{2-\e}}{(1-z)^{2+\e}}\,, \quad g(z)=-f(-z)=- \frac{(1-z)^{2-\e}}{(1+z)^{2+\e}},
$$
fail to satisfy the triangle inequality for $\|\cdot\|_{A^p}$.
\end{thm}
\begin{proof}
Note that $\e=1$ if and only if $p=1/2$; otherwise we have a whole interval to choose the value of $\e$. Also note that the numerator in the expression for $f$ is bounded in view of the condition $\e\le 1$ while $\e<2(1-p)/p$ implies that $p(2+\e)<2$. Thus, $f\in A^p$ by Lemma~\ref{lem-Ap}.
\par
We already know that $\|f\|_{A^p}=\|g\|_{A^p}$. Since
$$
 f(z)+g(z) = \frac{(1+z)^4-(1-z)^4}{(1-z^2)^{2+\e}} = \frac{8z (1+z^2)}{(1-z^2)^{2+\e}}\,,
$$
the desired inequality $\|f+g\|_{A^p}>\|f\|_{A^p}+\|g\|_{A^p}$ is equivalent to the statement that $\|f+g\|_{A^p}^p>2^p\|f\|_{A^p}^p$, that is,
$$
 2^{3p} \int_\D \frac{|z|^p |1+z^2|^p}{|1-z^2|^{(2+\e)p}}\,dA(z) > 2^p \int_\D \frac{|1+z|^{(2-\e)p}}{|1-z|^{(2+\e)p}}\,dA(z) = 2^{p+1} \int_\D \frac{|z|^2 |1+z^2|^{(2-\e)p}}{|1-z^2|^{(2+\e)p}}\,dA(z)\,,
$$
by Lemma~\ref{lem-cv} applied to the function $f$. This is clearly equivalent to
$$
 \int_\D \frac{|z|^p |1+z^2|^p \( 2^{2p-1} - |z|^{2-p} |1+z^2|^{(1-\e)p} \)}{|1-z^2|^{(2+\e)p}}\,dA(z) > 0.
$$
By our choice of $\e$ and restrictions on $p$, we have $(2p-1)-(1-\e)p = p +\e p -1\ge 0$ and $(1-\e)p\ge 0$, hence
$$
 2^{2p-1} - |z|^{2-p} |1+z^2|^{(1-\e)p} > 2^{2p-1} - 2^{(1-\e)p} \ge 0
$$
for all $z\in \D$, and the desired integral inequality follows.
\end{proof}
We now turn to the remaining range of exponents. To this end, the following simple inequality will be useful.
\begin{lem}\label{lem-elem}
If $a$, $b>0$ and $q>1$, then $|a^q-b^q|\ge |a-b|^q$.
\end{lem}
\begin{proof}
Without loss of generality, we may assume that $a\ge b>0$. Then, writing $x=a/b$, the inequality reduces to $(x^q-1)- (x-1)^q\ge 0$, for $x\ge 1$, and this is easily proved by elementary calculus since the function on the left-hand side is non-decreasing in $[1,+\infty)$.
\end{proof}
\par
Our next example covers the remaining range of values of $p$ and, actually, a somewhat larger interval of values.
\begin{thm}\label{thm-small-p}
Let $0<p<\frac12$ and define
$$
 f(z)=(1+z)^{4/p}\,, \quad g(z)=-f(-z)=-(1-z)^{4/p}\,,
$$
choosing the appropriate branch of the complex logarithm so that, say, $\log 1=0$. Then the functions $f$ and $g$ both belong to $A^p$ but 	 fail to satisfy the triangle inequality for $\|\cdot\|_{A^p}$.
\end{thm}
\begin{proof}
Again, it is clear that $\|f\|_{A^p}=\|g\|_{A^p}$. We can compute this value by using the formula \eqref{eqn-A2-norm} to obtain
$$
 \|f\|_{A^p}^p = \int_\D |1+2z+z^2|^2\,dA(z) = 1 + 2 + \frac13 = \frac{10}{3}\,.
$$
Next, we use the formulas defining the functions $f$ and $g$, the standard triangle inequality for complex numbers, then employ the elementary inequality from Lemma~\ref{lem-elem}
with
$$
 q=1/p\,, \quad a=|1+z|^4\,, \quad b=|1-z|^4\,,
$$
then some basic algebra of complex numbers, and afterwards express the integral obtained in polar coordinates and use Fubini's theorem to obtain:
\begin{eqnarray*}
 \|f+g\|_{A^p}^p &=&\int_\D |(1+z)^{4/p}-(1-z)^{4/p}|^p d A (z)
\\
 &\ge &  \int_\D \left| |1+z|^{4/p}-|1-z|^{4/p} \right|^p d A (z)
\\
 &\ge &  \int_\D \left| |1+z|^4-|1-z|^4 \right| d A (z)
\\
&=&  \int_\D \left| (1+|z|^2+2\re z)^2 - (1+|z|^2-2\re z)^2 \right| d A (z)
\\
 &=& 8 \int_\D (1+|z|^2) |\re z | d A (z)
\\
 &=& \frac{8}{\pi} \int_0 ^1 r^2 (1+r^2)\,dr \cdot 2 \int_{-\pi/2}^{\pi/2}\cos\t\,d\t
\\
 &=& \frac{2^8}{15\pi}
\\
 &>& 2^p \frac{10}{3}
\\
 &=& (\|f\|_{A^p}+\|g\|_{A^p})^p\,,
\end{eqnarray*}
whenever $p<\frac12$, as is easily seen by computation. Actually, the last inequality holds for a larger range of values of $p$ which we do not need for our purpose.
\end{proof}
\par
For the sake of simplicity, we have avoided discussing the weighted Bergman spaces $A^p_\a$ with standard radial weights. It does not appear too difficult to find related examples in such cases as well.
\par\smallskip
\textbf{Acknowledgments}. The authors would like to thank Ole F. Brevig and Raymond Mortini for their interest in the first draft of the paper and for some useful comments.
\medskip


\end{document}